\documentclass[11pt]{article}

\pdfoutput=1

\RequirePackage{my_packages}
\RequirePackage{my_theorems}
\RequirePackage{my_macros}

\usepackage[backend = biber,        
            language = english ,
            style = alphabetic-verb ,   
            giveninits = true,
            isbn = false,
            url = false,
            doi = false,
            sorting = nyt,
            maxnames = 6,
            backref=true
            ]{biblatex}

\renewbibmacro{in:}{%
  \ifentrytype{article}{}{%
  \printtext{\bibstring{in}\intitlepunct}}}
\AtEveryBibitem{%
  \clearfield{note}%
  \clearfield{edition}%
  \clearlist{language}
}

\DeclareNameAlias{labelname}{given-family}

\title{Orbit problems for free-abelian times free groups\\and related families
}
\author{André Carvalho$^*$, Jordi Delgado$^\dag$}
\date{\vspace{-7pt}
    $^*$Centro de Matemática, Faculdade de Ciências da Universidade do Porto\\%
    $^\dag$Departament de Matemàtiques, Universitat Politècnica de Catalunya\\[3ex]%
    \today
}

\newcommand{\Addresses}{{
  \bigskip
  \footnotesize

  André Carvalho\\\nopagebreak
  \textsc{Centro de Matemática, Faculdade de Ciências da Universidade do Porto\\
  R. Campo Alegre s/n\\ 4169-007 Porto, Portugal}\\\nopagebreak
  \url{andrecruzcarvalho@gmail.com}

  \medskip

  Jordi Delgado\\\nopagebreak
  \textsc{Departament de Matemàtiques, Universitat Politècnica de Catalunya\\\nopagebreak
  Escola Politècnica Superior d’Enginyeria de Manresa\\\nopagebreak 
  Av.~de les Bases de Manresa, 61, 73 \\\nopagebreak 
  08242 Manresa, Barcelona}\\\nopagebreak 
  \url{jorge.delgado@upc.edu }

}}

\hypersetup{final}  

\begin{document}

\maketitle

\begin{abstract}
We prove that the Brinkmann Problems (\BrP\ \& \BrCP) and the twisted conjugacy Problem (\TCP) are decidable for any endomorphism of a free-abelian times free (FATF) group~$\Fn \times \Zm$.
Furthermore, we prove the decidability of the two-sided Brinkmann conjugacy problem (\tsBrCP) for monomorphisms of FATF groups (and combine it with~$\TCP
$) to derive the decidability of the conjugacy problem 
for ascending HNN extensions of FATF groups.  
\end{abstract}

\bigskip



\vspace{15pt}

The family of free-abelian times free (FATF) groups, namely direct products of (finitely generated) free and free-abelian groups ($\Fn \times \Zm$) has elicited considerable interest in recent years. Their main point of interest lies in the fact that, despite being seemingly simple combinations of well-known groups (free and free-abelian), their behavior often cannot be reduced ---and could even be radically different--- from that of their constituent factors.

A typical example is the behavior of intersections: non-degenerated FATF groups do not satisfy the Howson property, whereas free and free-abelian groups certainly do (see~\cite{delgado_extensions_2017,Delgado_stallings_FTA_2022,delgado_intersection_2023} for results in this direction). Other non-trivial problems considered for this family include the study of relative order and spectrum \parencite{delgado_relative_2024}, compression and inertia \cite{roy_degrees_2021}, or different aspects of its group of automorphisms~\parencite{delgado_algorithmic_2013,roy_fixed_2020,carvalho_dynamics_2022}.

In \Cref{sec: FATF} we briefly introduce
free-abelian times free (FATF) groups $\Fn \times \ZZ^m$, present the (two types) of possible endomorphisms of these groups, their main properties, and the form of the powers of each type.

In \Cref{sec: Brinkmann}, we survey the decidability of the Brinkmann problems for free and free-abelian groups, and we combine these results with the findings in \Cref{sec: FATF} to deduce the decidability of Brinkmann Problems for endomorphisms of FATF groups.

Finally, in \cref{sec: CP HNN}, we consider a (two-sided) variation of the Brinkmann problems introduced by Logan \cite{logan_conjugacy_2023} in order to extend Bogopolski, Martino and Ventura's machinery in \cite{bogopolski_orbit_2010} to ascending HNN extensions of free groups. Using the same approach, we prove that the twisted conjugacy problem and the two-sided Brinkmann conjugacy problem are also decidable for FATF groups to finally reach the decidability of the conjugacy problem for ascending HNN extensions of FATF groups.

\medskip

We use lowercase boldface Latin font ($\vect{a},\vect{b},\vect{c},\ldots$) to denote elements in $\Zm$, and uppercase boldface Latin font to denote matrices ($\matr{A},\matr{B},\matr{C},\ldots $). Lowercase Greek letters ($\auti_i,\varphi, \chi,\ldots$) are employed to denote homomorphisms involving (typically free or free-abelian) factor groups, while uppercase Greek letters ($\Phi,\Psi,\Omega,\cdots$) are used for homomorphisms between FATF groups.

Additionally, both homomorphisms and matrices are assumed to act on the right; that is, we denote the image of the element $x$ by the homomorphism $\varphi$ as $(x)\varphi$ (or simply $x \varphi$) and the composition \smash{$A \overset{\varphi}{\to} B \overset{\psi}{\to} C$} is denoted as $\varphi \psi$.

The following is a convenient shorthand notation used throughout the paper to state decision problems. We write 
\begin{equation*}
\mathsf{D} \,\equiv\,
\mathsf{f}(\texttt{x})\?_{\mathsf{x} \,\in\, \mathsf{X}}
\end{equation*}
to denote the
the decision problem $\mathsf{D}$, with inputs $\mathsf{x}$ in $\mathsf{X}$ and outputs \yep/\nop\ depending on whether the condition $\mathsf{f}(\mathsf{x})$ holds.

\section{Free-abelian times free groups and their endomorphisms} \label{sec: FATF}

In this section we survey the basics about free-abelian
times free groups and their homomorphisms (see \cite{delgado_algorithmic_2013} and \cite{delgado_extensions_2017} for details and proofs).

A group $G$ is called \defin{free-abelian times free} (FATF) if it
is the direct product of a finitely generated free group $\Fn$ and a finitely generated free-abelian group $\ZZ^m$; that is if it 
admits a presentation of the form: 
\begin{equation} \label{eq: pres FATF}
\Pres{ x_1,\ldots,x_n,t_1,\ldots,t_m}{ t_it_j \,=\, t_jt_i,\, t_ix_k=x_kt_i \,, i,j\in[1,m], \ k\in[1,n] },
 \end{equation}
where $n,m \geq 0$.
Note that then, $\gen{x_1,\ldots,x_n} \isom \Fn$, $\gen{t_1,\ldots,t_m} \isom \ZZ^m$, and ${G \isom \Fn \times \ZZ^m}$. 
We generically denote FATF groups by $\GG$.

Note that, given a word in the generators, we can use the commutativity relations in~\eqref{eq: pres FATF} to orderly move all the $t_j$'s (say) to the right and describe the elements in $G$ by a normal form $u(x_1,\ldots,x_n)\, t_1^{a_1} \cdots t_m^{a_m}$, which we abbreviate as $u \t[a]$, where $u \in \Fn$ and $\vect{a} = (a_1,\ldots,a_m) \in \ZZ^m$ are called the free and free-abelian parts of $\fba{u}{a}$ respectively.
Note that then, $(u\t^\mathbf{a})(v\t^\mathbf{b})=uv \, \t^\mathbf{a+b}$.

We denote by $\pi\colon \FTA \to \Fn$, $u\t^\textbf{a}\mapsto u$ and by $\tau\colon \FTA \to \ZZ^m$, $u\t^\textbf{a}\mapsto \textbf{a}$, the natural epimorphisms to the free and the free-abelian part respectively.


Below, we summarize 
several fundamental results regarding endomorphisms of FATF groups to be used in later sections of this paper.  The starting point is the description and classification of the endomorphisms of FATF groups which was initially obtained in \cite{delgado_algorithmic_2013,delgado_extensions_2017} and later generalized to FABF groups ($\Fn \ltimes \Zm$) in \cite{carvalho_free-abelian_2024}.

\begin{prop}
\label{prop: endos FATF}
Let $m\geq 1,n\geq 2$. The following is a complete list of the endomorphisms of~\,$\GG = \Fn \times \ZZ^m$:
\begin{enumerate} [dep]
\item Endomorphisms of type I:
\begin{equation} \label{eq: hom FATF I}
\Phi_{\varphi,\mathbf{Q,P}}\colon
 u\t^{\vect{a}} \mapsto
 u\varphi \, \t^{\vect{a}\matr{Q}+ \vect{u} \matr{P}} \,,
\end{equation}
where $\varphi \in \End(\Fn)$, $\mathbf{Q}\in \mathcal{M}_{m}(\ZZ)$,
$\mathbf{P}\in \mathcal{M}_{n\times m}(\ZZ)$, and $\vect{u} = u\ab \in \ZZ^n$.
\item  \label{item: type II}
Endomorphisms of type II:
\begin{equation} \label{eq: hom FATF II}
    \Phi_{w,\mathbf{r,s,Q,P}}\colon
 u\t^{\vect{a}} \mapsto
w^{\vect{a} \vect{r}^{\,\T} +\vect{u}\vect{s}^\T} \, \t^{\vect{a}\matr{Q}+\vect{u}\matr{P}}
\end{equation}
where $w\in \Fn \setmin\{\trivial\}$ is not a proper power, $\mathbf{Q}\in \mathcal{M}_{m}(\ZZ)$, $\mathbf{P}\in
\mathcal{M}_{n\times m}(\ZZ)$, $ \vect{r}\in \ZZ^m\setmin\{\vect{0}\} $,
$\mathbf{s}\in \ZZ^n$, and $\vect{u} = u\ab \in \ZZ^n$. \qed
\end{enumerate} 
\end{prop}

\begin{cor}\label{cor: injective bijective}
An endomorphism $\Phi \in \End(\Fn \times \Zm)$ is injective (\resp surjective or bijective) if and only if it is of type I, $\varphi$ is injective (resp. bijective) and $\det(\matr{Q}) \neq 0$ (\resp $\det(\matr{Q}) = \pm 1$). \qed
\end{cor}

An immediate consequence of our description is the well-known result below. In \cite{green_graph_1990} and \cite{humphries_stephen_p._representations_1994}
it is proved that finitely generated partially commutative groups (which include FATF groups) are residually finite and, thus, Hopfian.
\begin{cor}
 Nontrivial FATF groups are Hopfian and not co-Hopfian. \qed
\end{cor}
The respective expressions for composition and powers of endomorphisms of each type follow easily from the expressions \eqref{eq: hom FATF I} and \eqref{eq: hom FATF II}.

\begin{prop} \label{prop: powers FATF I}
Let $\Phi_{\varphi,\matr{Q},\matr{P}}, \Phi_{\varphi',\matr{Q'},\matr{P'}} \in \End_{I}(\Fn \times \Zm)$ and let $k>0$. Then, 
\begin{align}
    \Phi_{\varphi,\matr{Q},\matr{P}} \
    \Phi_{\varphi',\matr{Q'},\matr{P'}}
    &\,=\,
    \Phi_{\varphi\!\varphi',\,\matr{Q}\matr{Q'},\,\matr{P} \matr{Q'} + \varphi\ab \matr{P'}}\\
    (\Phi_{\!\varphi,\matr{Q},\matr{P}})^k
    &\,=\, 
    \Phi_{\varphi^k,\,\matr{Q}^k, \, \matr{P}^{(k)}}, \label{eq: powers hom FATF I}
\end{align}
where $\varphi \ab \colon \Zn \to \Zn$ is the abelianization of $\varphi$, and
$\matr{P}^{(k)} = \sum_{i=1}^{k} (\varphi \ab)^{i-1} \matr{P} \matr{Q}^{k-i}.$ \qed
\end{prop}

Finally, we highlight two important properties that combine into a compact description of the (powers of) endomorphisms of type II.

\begin{rem}
Let $\Phi \in \Endii (\Fn \times \Zm)$ be an endomorphism of type II, and let $u,v \in \Fn$. Then,
\begin{enumerate}[ind]
    \item if $u\ab =\vect{u} = \vect{v} = v\ab$ then $\Phi(u\t[a]) = \Phi(v\t[a])$; that is, $\Phi$ factors through the abelianization $\Fn \times \Zm \ni u\t[a] \to (\vect{u} , \vect{a}) \in \Zn \times \Zm$.
    \item the image $\im (\Phi) \leqslant \gen{w} \times \Zm \isom \ZZ^{m+1}$. In particular, $\im (\Phi)$ is abelian.
\end{enumerate}

Note that these properties allow to understand endomorphisms of type II essentially as linear maps. Concretely, if we express the images of $\Phi$ \wrt the (ordered) free-abelian basis $W = (w,t_1,\ldots ,t_m)$  of 
(the free-abelian group)
$\gen{w} \times \Zm$, then any homomorphism \eqref{eq: hom FATF II} of type II can be decomposed as:
\begin{equation} \label{eq: hom FATF II concise}
\begin{array}{rcccl}  \Phi_{w,\mathbf{r,s,Q,P}}\colon \Fn \times \ZZ^m & \overset{\rho}\to & \Zn \times \Zm & \to &\ZZ^{m+1}\\
    u \t[a] & \mapsto & (\vect{u},\vect{a}) & \mapsto & (\vect{u},\vect{a})
\left(
\begin{smallmatrix}
    \vect{s}^{\,\T} & \matr{P}\\
    \vect{r}^{\,\T} & \matr{Q}
\end{smallmatrix}
\right),
\end{array}
\end{equation}
where $\vect{u} = u \ab$ is the abelianization of $u \in \Fn$.
That is, endomorphisms of type II are the result of abelianizing the free part and then applying a linear transformation (and interpret the result in basis $W$).
\end{rem}
Applying this scheme recursively we obtain a concise description for the composition and powers of endomorphisms of type II. 

\begin{prop} \label{prop: comp hom FATF II}
Let $\Phi_{1} = \Phi_{w_{1},\mathbf{r_{1},s_{1},Q_{1},P_{1}}}$
and $\Phi_{2} = \Phi_{w_{2},\mathbf{r_{2},s_{2},Q_{2},P_{2}}}$
be endomorphisms of $\Fn \times \ZZ^m$ of type~II,
and let $u \t[a] \in \Fn \times \Zm$.
Then, 
\begin{equation}
(u \t[a]) \Phi_1 \Phi_2
\,=\,
(\vect{u}, \vect{a})
    \left(
\begin{smallmatrix}
    \vect{s_{1}^{\,\T}} & \matr{P_{1}}\\
    \vect{r_{1}^{\,\T}} & \matr{Q_{1}}
\end{smallmatrix}
\right)
\left(
\begin{smallmatrix}
    \vect{w_{1}}\vect{s_{2}^{\,\T}} & \vect{w_1} \matr{P_{2}}\\
    \vect{r_{2}^{\,\T}} & \matr{Q_{2}}
\end{smallmatrix}
\right),
\end{equation}
where $\vect{u}, \vect{w_1} \in \Zm$ denote the respective abelianizations of $u, w_1 \in \Fn$, and the image is written in basis $W_2 = (w_2,t_1,\ldots ,t_m)$.
\end{prop}

\begin{proof}
    It is enough combine the expressions \eqref{eq: hom FATF II} and \eqref{eq: hom FATF II concise} for endomorphisms of type II taking into account that the abelianization of a power $w^l \in \Fn$ is $(w^l)\ab = l\, w\ab = l \vect{w} \in \Zn$. Then,
    \begin{align*}
    (u \t[a]) \Phi_1 \Phi_2
    &\,=\,
    \big(w_1^{\vect{u} \vect{s_{1}^{\,\T}} + \vect{a} \vect{r_{1}^{\,\T}}} \ \t[uP_1 + aQ_1] \big) \Phi_2\\
    &\,=\,
    \big(w_1^{\vect{u} \vect{s_{1}^{\,\T}} + \vect{a} \vect{r_{1}^{\,\T}}} \ \t[uP_1 + aQ_1] \big) \rho
    \left(
\begin{smallmatrix}
    \vect{s_{2}^{\,\T}} & \matr{P_{2}}\\
    \vect{r_{2}^{\,\T}} & \matr{Q_{2}}
\end{smallmatrix}
\right)\\
&\,=\,
    \big((\vect{u} \vect{s_{1}^{\,\T}} + \vect{a} \vect{r_{1}^{\,\T}}) \vect{w_1} \, , 
    \vect{u} \matr{P_1} + \vect{a}\matr{Q_1}\big)
    \left(
\begin{smallmatrix}
    \vect{s_{2}^{\,\T}} & \matr{P_{2}}\\
    \vect{r_{2}^{\,\T}} & \matr{Q_{2}}
\end{smallmatrix}
\right)\\
&\,=\,
    (\vect{u} \vect{s_{1}^{\,\T}} + \vect{a} \vect{r_{1}^{\,\T}} ,
    \vect{u} \matr{P_1} + \vect{a}\matr{Q_1})
    \left(
\begin{smallmatrix}
    \vect{w_1} \vect{s_{2}^{\,\T}} & \vect{w_1} \matr{P_{2}}\\
    \vect{r_{2}^{\,\T}} & \matr{Q_{2}}
\end{smallmatrix}
\right)\\
&\,=\,
    (\vect{u}, \vect{a})
    \left(
\begin{smallmatrix}
    \vect{s_{1}^{\,\T}} & \matr{P_{1}}\\
    \vect{r_{1}^{\,\T}} & \matr{Q_{1}}
\end{smallmatrix}
\right)
\left(
\begin{smallmatrix}
    \vect{w_{1}}\vect{s_{2}^{\,\T}} & \vect{w_1} \matr{P_{2}}\\
    \vect{r_{2}^{\,\T}} & \matr{Q_{2}}
\end{smallmatrix}
\right),
\end{align*}
as claimed.
\end{proof}
The expression for powers of endomorphisms of type II follows easily by induction.

\begin{prop} \label{prop: powers hom FATF II}
Let $\Phi = \Phi_{w,\vect{r},\vect{s},\matr{P},\matr{Q}}$ be an endomorphism of $\Fn \times \Zm$ of type II, and let $u \t[a] \in \Fn \times \Zm$. Then, for every $k \geq 1$,
\begin{equation} \label{eq: powers hom FATF II}
    (u\t[a])\Phi^{k}
    \,=\,
    (\vect{u},\vect{a}) 
    \left(
\begin{smallmatrix}
    \vect{s}^{\,\T} & \matr{P}\\
    \vect{r}^{\,\T} & \matr{Q}
\end{smallmatrix}
\right)
\left(
\begin{smallmatrix}
    \vect{w} \vect{s}^{\,\T} & \vect{w} \matr{P}\\
    \vect{r}^{\,\T} & \matr{Q}
\end{smallmatrix}
\right)^{k-1},
\end{equation}
where $\vect{u}, \vect{w} \in \Zm$ denote the respective abelianizations of $u, w \in \Fn$, and the image is written in basis $W = (w,t_1,\ldots ,t_m)$.\qed
\end{prop}

\section{The Brinkmann problem}\label{sec: Brinkmann}
Orbit problems (consisting in deciding whether two objects belong to the same orbit under certain kind of transformations)
constitute an important family of algorithmic problems in group theory.
They encompass many classical problems, such as the \defin{conjugacy problem} (where the objects are group elements and the transformations are inner automorphisms) or the Whitehead Problems (with variants for elements, subgroups and different kinds of endomorphisms).

In this section we deal with two types of orbit problems involving monogenic subgroups of transformations, initially considered by Brinkmann in \cite{brinkmann_detecting_2010} for automorphisms of the free group $\Fn$. In both cases, $G$ is supposed to be a group given by a finite presentation, and $\mathcal{T} \subseteq \End(G)$ is a family of endomorphisms
(given as a set of words representing the images of the generators).

\begin{named}[Brinkmann Problem for $G$ \wrt $\mathcal{T}$, $\BrP_{\mathcal{T}}(G)$]
Decide,
given an endomorphism $\varphi \in \mathcal{T}$ and two elements $x,y \in G$,
 whether there exists some $k \in \NN$ such that $(x)\varphi^k = y$.
 \end{named}

\begin{named}[Brinkmann Conjugacy Problem for $G$ \wrt $\mathcal{T}$, $\BrCP_{\mathcal{T}}(G)$]
Decide,
given an endomorphism $\varphi \in \mathcal{T}$ and two elements $x,y \in G$,
 whether there exists some $k \in \NN$ such that $(x)\varphi^k$ and $y$ are conjugate in $G$.
 \end{named}

We (generically) abbreviate `Brinkmann Problem' as \BrP, and
`Brinkmann Conjugacy Problem' as \BrCP.
In the table below we summarize the shorthand notations we use for the different versions of the Brinkmann Problem considered in this paper.

\begin{table}[h]
\centering
\begin{tabular}{llll}
& $\Aut(G)$ & $\Mon(G)$ & $\End(G)$ \\ \toprule
 Brinkmann Problem (\BrP) & $\BrPa(G)$ & $\BrPm(G)$ & $\BrPe(G)$ \\ 
 Brinkmann Conjugacy Problem (\BrCP) & $\BrCPa(G)$ & $\BrCPm(G)$ & $\BrCPe(G)$
\end{tabular}
 \caption{Variants of the Brinkmann Problem}
\end{table}

It is clear that
\[
\WP(G) \preceq \BrPa(G) \preceq \BrPm(G) \preceq \BrPe(G)\]
and 
\[
\CP(G) \preceq \BrCPa(G) \preceq \BrCPm(G) \preceq \BrCPe(G). \footnote{We write  $\mathsf{P} \preceq \mathsf{Q}$ to express that $\mathsf{P}$ is Turing reducible to $\mathsf{Q}$.}
\]

\begin{rem} \label{rem: yes outputs}
Note that the decidability of Brinkmann's problems
immediately allows us to theoretically compute a witness exponent $n$ in case it exists: if the answer to the corresponding Brinkmann problem is \yep, then it is enough to keep enumerating the successive images $((x)\varphi^n)_{n \geq 0}$ and use the positive part of word problem (\resp conjugacy problem) to check for a guaranteed match with $y$. 

Moreover, if $\WP(G)$ (\resp $\CP(G)$) is decidable
---as it happens whenever $\id_G\in\mathcal{T}$---
we can perform the previous search sequentially, which allows us to compute minimal witnesses, and from them the full set of witnesses, in the sense made precise below. 
\end{rem}

\begin{defn}
Let $G$ be a group, let $x,y \in G$, and let $\varphi \in \End (G)$. Then, the set of \defin{$\varphi$-logarithms} (\resp \defin{$\varphi\cj$-logarithms}) of $y$ in base $x$ is
\begin{equation*}
    \varphi\text{-\!}\log_x(y)
    \,=\,
    \set{k\geq 0 \st (x)\varphi^k = y}
    \qquad
    \text{(\resp}\varphi\cj\text{-\!}\log_x(y)
    \,=\,
    \set{k\geq 0 \st (x)\varphi^k \conj y}\text{)}
\end{equation*}
Note that $0 \in \varphi\text{-\!}\log_x(y)$ 
(\resp $0 \in \varphi\cj\text{-\!}\log_x(y)$)
if and only if
$x=y$ (\resp $x\conj y$); and then
$
\varphi\text{-\!}\log_x(x) = p\NN$ 
(\resp $
\varphi\cj\text{-\!}\log_x(x) = p\NN$), for some $p \in \NN$, which is called the 
\defin{$\varphi$-period} of~$x$
(\resp
\defin{$\varphi\cj$-period} of~$x$).
In particular we say that the $\varphi$-period
(\resp $\varphi\cj$-period)
of $x$ is $0$ if~$(x)\varphi^k \neq x$
(\resp $(x)\varphi^k \nconj x$) for all $k>0$. 
\end{defn}


In \cite{carvalho_free-abelian_2024}, the authors proved the following general lemmas about the structure and computability of $\varphi$-logarithms and $\varphi\cj$-logarithms:
\begin{lem} \label{lem: philog form}
Let $G$ be a group, let $x,y \in G$, and let $\varphi \in \End (G)$. Then, 
\begin{enumerate}[ind]
    \item \label{item: philog}
    either $\varphi\text{-\!}\log_x(y) = \varnothing$ or $\varphi\text{-\!}\log_x(y) = k_0 + p\NN$,
    where $k_0 = \min (\varphi\text{-\!}\log_x(y))$, and $p$ is the $\varphi$-period of $y$.
    \item \label{item: cphilog}
    either $\varphi\cj\text{-\!}\log_x(y) = \varnothing$ or $\varphi\cj\text{-\!}\log_x(y) = k_0 + p\NN$,
    where $k_0 = \min (\varphi\cj\text{-\!}\log_x(y))$, and $p$ is the $\varphi\cj$-period of $y$. \qed
\end{enumerate}
\end{lem}



\begin{lem} \label{lem: logs computable}
    If $\WP(G)$ (\resp $\CP(G)$) is decidable then the decidability of $\BrP(G)$ (\resp $\BrCP(G)$) allows to compute $\varphi\text{-\!}\log_x(y)$ (\resp $\varphi\cj\text{-\!}\log_x(y)$) for every $x,y \in G$. \qed
\end{lem}

In the remainder of this section we study the Brinkmann's problems $\BrP$ and $\BrCP$ in the context of \FATF\ groups. 

In \cite{carvalho_free-abelian_2024} we extend the description of endomorphisms from FATF groups to free-abelian by free (FABF) groups (of the form $\Fn \ltimes \Zm$). It turns out that the obtained classification scheme naturally extends that in \Cref{prop: endos FATF}, splitting again in types I and II, which when restricted to the direct product case correspond exactly to types I and II for FATF groups. Moreover, in this work, a general argument is provided to prove the decidability of \BrP\ for the class of all endomorphisms of FABF (and hence of FATF) groups of type I.
Accordingly, below we only provide a proof for the decidability of $\BrP(\Fn \times \Zm)$ for the remaining case of endomorphisms (of FATF groups) of type II. On the other side, $\BrCPe(\Fn \times \Zm)$ must be considered in full generality, since the $\CP$ (and hence the $\BrCPe$) is known to be undecidable for FABF groups.

A first natural step is to consider the aforementioned problems on its factors, namely for free and free-abelian groups.
We recall that Brinkmann's original result in \cite{brinkmann_detecting_2010} is that both $\BrPa(\Fn)$ and $\BrCPa(\Fn)$ are algorithmically decidable. The decidability of the general case (of endomorphisms of~$\Fn$) was finally obtained in \cite{carvalho_decidability_2024} building on the work of  Logan \cite{logan_conjugacy_2023}, who proves the decidability of  $\BrPm(\Fn)$ and $\BrCPm(\Fn)$ (among other variations) using recent techniques by Mutanguha extending to endomorphisms the computability of the fixed subgroup of endomorphisms of $\Fn$ \cite{mutanguha_constructing_2022}. 

\begin{thm} \label{thm: BrPe free}
Both $\BrPe(\Fn)$ and $\BrCPe(\Fn)$ are algorithmically decidable. \qed
\end{thm}

The corresponding result for free-abelian groups was obtained  in \cite{kannan_polynomial-time_1986}, where \citeauthor{kannan_polynomial-time_1986} provide  a polynomial algorithm to solve $\BrPe(\QQ^m)$. 
As detailed in \cite{carvalho_free-abelian_2024} it is immediate to extend this result to the family of affine transformations
of $\Zm$, which we denote by $\Aff(\Zm)$.

\begin{thm}[\citenr{kannan_polynomial-time_1986}]
\label{thm: KannanLipton}
The Brinkmann problem $\BrP_{\!\operatorname{Aff}}(\Zm)$ is algorithmically decidable. \qed 
\end{thm}

Below, we combine \Cref{thm: KannanLipton,thm: BrPe free} with our description of powers of endomorphisms of type II 
to prove the decidability of $\BrP$ for this latest family. This, together with \cite[Theorem 4.11]{carvalho_free-abelian_2024} combine in the result  below.

\begin{thm} \label{thm: BrPe FATF}
    Let $n,m\in \NN$. Then, the Brinkmann Problem $\BrPe(\Fn \times \Zm)$ is decidable.
\end{thm}

\begin{proof}
    Since decidability of $\BrP$ for endomorphisms of type I was already obtained in \cite{carvalho_free-abelian_2024} for the more general family of FABF groups (of the form $\Fn \ltimes \Zm$), below we only prove the decidability of $\BrP(\Fn \times \Zm)$ for the remaining case of endomorphisms (of FATF groups) of type II.
    That is, given $u\t[a],v\t[b] \in \GG = \Fn \times \Zm$ and $\Phi = \Phi_{w,\vect{r},\vect{s},\matr{Q},\matr{P}} \in \Endii(\GG)$, our goal is to decide whether
    \begin{equation} \label{eq: BrPII(G)}
    \exists k \in \NN \text{ such that } (u \t[a]) \Phi^k = v \t[b].
    \end{equation}
  It is clear from \Cref{prop: powers hom FATF II} that the free part of $(u \t[a]) \Phi^k$  must belong to the cyclic subgroup $\gen{w} \leqslant \Fn$. Hence, we start using the decidability of $\MP(\Fn)$ to check whether $v \in \gen{w}$. If the answer is \nop, then the answer to \eqref{eq: BrPII(G)} is \nop\ as well. Otherwise, $u \in \gen{w}$ and we can compute the integer $l \in \ZZ$ such that $v = w^l$. That is $v \t[b] \in \gen{w} \times \ZZ^m$, and the expression of $v \t[b]$ in base $W = (w,t_1,\ldots,t_n)$ is $(l,\vect{b}) \in \ZZ^{m+1}$. Then, using the expression \eqref{eq: powers hom FATF II} for powers of endomorphisms of type II, the condition \eqref{eq: BrPII(G)} takes the form:
  \begin{equation}
    \exists k \in \NN \text{ such that }
    (\vect{u},\vect{a}) 
    \left(
\begin{smallmatrix}
    \vect{s}^{\,\T} & \matr{P}\\
    \vect{r}^{\,\T} & \matr{Q}
\end{smallmatrix}
\right)
\left(
\begin{smallmatrix}
    \vect{w} \vect{s}^{\,\T} & \vect{w} \matr{P}\\
    \vect{r}^{\,\T} & \matr{Q}
\end{smallmatrix}
\right)^{k-1} = (l,\vect{b}) \,,
    \end{equation}
which (since all the involved data is deducible from the input) is just an instance of $\BrP(\ZZ^{m+1})$, and hence decidable by \Cref{thm: KannanLipton}.
\end{proof}

The technical lemma below will be useful in order to 
study the $\BrCP$ for endomorphisms of $\GG = \Fn \times \Zm$ of type I. Recall that 
if $\Phi = \Phi_{\varphi,\matr{Q},\matr{P}} \in \Endi(\GG)$, then $(u \t[a])\Phi = u \varphi \ \t^{\vect{a} \matr{Q} + u\ab \matr{P}}$,
where $u\ab \in \ZZ^n$ denotes the abelianization of $u \in \Fn$.


\begin{lem}\label{lem: powers conj periodic affine FATF}
Let $\GG$ be a $\FATF$ group, $\Phi_{\varphi,\matr{Q},\matr{P}}\in \Endi(\GG)$, and $k,p\in \NN$ and $u\in \Fn$ be such that 
$u \varphi^{k+p} \conj u \varphi^{k}$.
Then, for all $\lambda\geq 1$ and $\vect{a}\in \Zm$, 
\begin{equation} \label{eq: powers affine FATF}
(u\ta{a})\Phi^{k+\lambda p}\tau=((u\ta{a})\Phi^k\tau)\widetilde{T}^\lambda
\end{equation}
where $\widetilde{T}:\Zm\to\Zm$ is the affine transformation
$\vect{x}\mapsto \vect{x
}\matr{Q}^p
+ (u \varphi^{k})\ab \, \matr{P}^{(\!p\!)}$.
\end{lem}
\begin{proof}
Since $u\varphi^k\varphi^p\sim u\varphi^k$ then, for all $\lambda\in\NN$, $u\varphi^{k+\lambda p}\sim u\varphi^k$, and hence $(u\varphi^{k+\lambda p})\ab =  (u\varphi^k) \ab$. In particular, from \eqref{eq: hom FATF I} and \eqref{eq: powers hom FATF I}:
\[
\begin{aligned} \label{eq: periodic affine}
(u\ta{a})\Phi^{k + \lambda p}\tau
&\,=\,
(u\ta{a})\Phi^{k + (\lambda-1) p} \Phi^{p}\tau\\
&\,=\,
(u\ta{a})\Phi^{k + (\lambda-1) p} \tau \,\matr{Q}^p +
(u \varphi^{k+ (\lambda - 1) p})\ab \, \matr{P}^{(p)}\\
&\,=\,
(u\ta{a})\Phi^{k + (\lambda-1) p}  \tau \,\matr{Q}^p +
(u \varphi^{k})\ab \, \matr{P}^{(p)}
\end{aligned}
\]
and the claimed result follows.
\end{proof}





\begin{thm} \label{thm: BrCPm(FATF)}
Let $n,m\in \NN$. Then, the Brinkmann Conjugacy Problem $\BrCPe(\Fn \times \Zm)$ is decidable.
\end{thm}

\begin{proof}
Let $u\ta{a},v\ta{b}$ be two elements in $\GG = \Fn\times\Zm$
and let $\Phi \in \End(\GG)$.
Our goal is to decide whether there exists some $k \in \NN$ such that $(u\t[a])\Phi^k \conj v \t[b]$.
Since $u \t[a]$ and $v \t[b]$ are conjugate in $\Fn \times \Zm$ if and only if $u \conj v$ in $\Fn$ and $\vect{a} = \vect{b}$, splitting the free and free-abelian parts our problem becomes to decide whether there exists some $k \in \ZZ$ such that the two conditions below hold:
\begin{align}[left=\empheqlbrace]
    (u \t[a]) \Phi^k \pi &\,\sim\, v\label{eq: BrCP FATF fr}\\\hspace{2pt}
    (u\ta{a}) \Phi^k \tau &\,=\, \vect{b} . \label{eq: BrCP FATF ab}
\end{align}

We distinguish two cases depending on whether $\Phi$ is of type I or II (see \Cref{prop: endos FATF}).
If  
$\Phi = \Phi_{\varphi,\matr{Q},\matr{P}}
$ is of type~I
(that is, $\varphi$, $\matr{Q}$ and $\matr{P}$ are part of the input), then, according to the expression \eqref{eq: powers hom FATF I} for powers of endomorphisms of type I our problem consists in deciding whether there is some $k\in \NN$ such that the following two conditions hold:
\begin{align}[left=\empheqlbrace]
    (u) \varphi^k &\,\sim\, v \label{eq: }\\\hspace{2pt}
    (u\ta{a}) \Phi^k \tau 
    &\,=\, \vect{b} 
\end{align}
We start using the decidability of $\BrCPe(\Fn)$
to check whether there exists some $k \in \NN$ such that $(u) \varphi^k \sim v$. If the answer is \nop, then we stop the process and the answer to $\BrCPe(\GG)$ is also \nop. If the answer is \yep, then in view of \Cref{{lem: logs computable}}, we can assume that $\BrCPe(\Fn)$ outputs $k_0$ and $p$ such that $u\varphi^k\sim v$ if and only if $k\in k_0+p\NN$.
Therefore, our problem translates to deciding whether 
\begin{equation} \label{eq: BrPa tau}
    \exists \lambda \in \NN \st (u\ta{a}) \Phi^{k_0 + \lambda p} \tau \,=\, \vect{b} \,,
\end{equation} 
which according to \Cref{lem: powers conj periodic affine FATF},  can be rewritten as:
\begin{equation}
\exists \lambda \in \NN \st ((u\ta{a}) \Phi^{k_0} \tau) \, \widetilde{T}^{\lambda} \,=\, \vect{b} ,
\end{equation}
where  $\widetilde T$ is the affine transformation $\widetilde T \colon \ZZ^m \to \ZZ^m$, $\vect{x} \mapsto \vect{x} \matr{Q}^{p} + (v\varphi^{k_0})\ab\matr{P}^{(p)}$. 
Again, this is just an instance of the Brinkmann problem for affine transformations, and hence decidable by \Cref{thm: KannanLipton}.

Suppose now that the input endomorphism $\Phi$ is of type II, namely $\Phi = \Phi_{w,\vect{r},\vect{s},\matr{Q},\matr{P}} \in \Endii(\GG)$. We start considering separately the case $k=0$, namely we check whether $u \t[a] \conj v \t[b]$ (which is equivalent to check whether $u \conj v$ in $\Fn$ and $\vect{a} = \vect{b}$ and hence decidable).
If the answer is \yep, then the answer to $\BrCPe(\GG)$ is also \yep. Otherwise, we can restrict conditions \eqref{eq: BrCP FATF fr} and \eqref{eq: BrCP FATF ab} to $k \geq 1$. Note that in this case $(u \t[a]) \Phi^k \pi \leqslant \gen{w}$ (see \Cref{prop: powers hom FATF II}) and hence $v$ must be conjugate to some power of $w$. 
So, the next step is to check whether 
$v \conj w^l$ (or equivalently whether $\cred{v} \conj \cred{w}^l$, where $\cred{u}$ denotes the cyclic reduction of $u$) for some $l \in \ZZ$. We start by computing the ratio $l= 
|\cred{v}| / |\cred{ w}|
$ of the lengths
of the cyclic reductions of $v$ and $w$. Since $|\cred{w}^l|=l|\cred{ w}|$ and two conjugate words must have cyclic reductions of the same length,
if $l$ is not an integer, then the answer to $\BrCPe(\GG)$ is \nop; otherwise, $l$ is the only candidate and it is enough to use $\CP(\Fn)$ to decide whether $v\sim w^l$. If the answer is \nop\ then the answer to $\BrCPe(\GG)$ is again \nop; otherwise we have computed (the only) $l\in \NN$ such that $v\sim w^l$, and hence condition \eqref{eq: BrCP FATF fr} takes the form: $(u)\Phi^k \pi \conj w^l$.

Joining with condition \eqref{eq: BrCP FATF ab} and 
according to \Cref{prop: powers hom FATF II}, it only remains to check whether 
\begin{equation} \label{eq: BrCPII}
    \exists k \in \NN \text{ such that }
    (\vect{u},\vect{a}) 
    \left(
\begin{smallmatrix}
    \vect{s}^{\,\T} & \matr{P}\\
    \vect{r}^{\,\T} & \matr{Q}
\end{smallmatrix}
\right)
\left(
\begin{smallmatrix}
    \vect{w} \vect{s}^{\,\T} & \vect{w} \matr{P}\\
    \vect{r}^{\,\T} & \matr{Q}
\end{smallmatrix}
\right)^{k-1}
= \left(
    l , \vect{b} 
\right)\,,
    \end{equation}
which is just an instance of $\BrP(\ZZ^{m+1})$, and hence decidable by \Cref{thm: KannanLipton}. Then,  the answer to $\BrCPe(\GG)$ is \yep\ if and only if the answer to \eqref{eq: BrCPII} is \yep, and the proof is concluded.
\end{proof}

\section{The conjugacy problem for ascending HNN extensions} \label{sec: CP HNN}
Besides its intrinsic interest, the Brinkman Problem turns out to be crucially related to other important algorithmic problems. Together with the Twisted Conjugacy Problem (\TCP, defined below), it is involved in (a restricted variant of) the conjugacy problem proving machinery developed by \citeauthor{bogopolski_orbit_2010} in the remarkable paper~\cite{bogopolski_orbit_2010}. Concretely, they prove that if a group $G$ fits in the middle of a short exact sequence of the form
$   \trivial \to F \to G \to H \to \trivial$, where  the twisted conjugacy problem~$\TCPa(F)$ and the conjugacy problem $\CP(H)$ are decidable then (with some extra algorithmic assumptions) $\CP(G)$ is decidable if and only if the action subgroup $A_G = \{\gamma_g \colon f \mapsto g^{-1} f g \mid g\in G \} \leqslant \Aut(F)$ is orbit-decidable. 
We will further refer to this result as the BMV Theorem.
We recall that the Brinkmann problem is nothing more that the orbit decidability of cyclic subgroups, an hence one of the main ingredients for the solvability of the conjugacy problem of cyclic extensions.

A classic application of BMV Theorem is the study of the \CP\ of FABF groups (of the form~$\Fn \ltimes \Zm$),
which, using this result, is easily seen to be undecidable in general.
Another consequence (and precursor) of BMV Theorem, previously obtained in \cite{bogopolski_conjugacy_2006},
is the decidability of $\CP(\ZZ \ltimes \Fn)$,
which is reduced to that of $\BrCPa(\Fn)$ and $\TCPa(\Fn)$. 

In order to extend this last result to ascending HNN extensions of free groups, A.~Logan considers in \cite{logan_conjugacy_2023} the variations of Brinkmann's problems stated below for a generic group~$G$ given by a finite presentation \wrt a family of transformations $\mathcal{T} \subseteq \End(G)$.

\begin{named}[Two-sided Brinkmann Problem for $G$ \wrt $\mathcal{T}$]
\[
\tsBrP[T](G)
\,\equiv\,
     \exists (r,s)\in \NN^2 \st x\varphi^r = y\varphi^s\, \?_{\subalign{
    &\varphi\,\in\,\mathcal{T},\\
    & x,y \,\in\,G}}
\]
\end{named}

\begin{named}[Two-sided Brinkmann Conjugacy Problem for $G$ \wrt $\mathcal{T}$]
\[
\tsBrCP[T](G)
\,\equiv\,
    \exists (r,s)\in \NN^2 \st x\varphi^r \sim y\varphi^s\, \?_{\subalign{
    &\varphi\,\in\,\mathcal{T},\\
    & x,y \,\in\,G}}
\]
\end{named}

We write $\tsBrPa(G) = \tsBrP[\Aut(\mathsf{G})]$,  $\tsBrPm(G) =\tsBrP[\Mon(\mathsf{G})]$, and  $\tsBrPe(G) = \tsBrP[\End(\mathsf{G})]$, and similarly for $\tsBrCP$.

As it happens with the standard Brinkmann Problems (and for the same reasons), 
the decidability of 
the  2-sided Brinkmann Problems guarantees the computability of a pair o witnesses $(r,s)$ in case they exist.

We recall that, if $G$ is a group and $\Phi\in \Mon(G)$, then the \emph{ascending HNN extension of $G$ induced by $\Phi$} is the group with relative presentation 

\begin{equation} \label{eq: presentation AHNN}
    G \ast_\Phi
    \,=\,
    \Pres{\hspace{-3pt}
    \begin{array}{ll}
    G,
    x
    \end{array}
    }
    {
    \begin{array}{ll}
    x^{-1}gx=\Phi(g), g\in G
    \end{array} \hspace{-5pt}
    }
\end{equation}
and with normal forms given by $x^i g x^{-j}$, with $i,j\in \NN_0$ and $g\in G.$

The purpose of this subsection is to use the same approach to prove that the conjugacy problem is decidable for ascending HNN extensions of $\Fn \times \Zm$.


\begin{rem}
In \cite{bogopolski_conjugacy_2006}, the authors reduce the decidability of the conjugacy problem in free-by-cyclic groups to the decidability of the $\TCPa$ and $\BrCPa$ in free groups. Following their proof step by step, we can see that this machinery works for any group, in the sense that  $\CP(G\rtimes \ZZ)$ can be reduced to $\TCPa(G)$ and $\BrCPa(G)$, for all groups $G$. Similarly, following Logan's proof of the conjugacy problem for ascending HNN extensions of free groups \cite{logan_conjugacy_2023}, it follows that this machinery can also be used to prove the conjugacy problem for ascending HNN extensions of $G$ for all groups $G$.
\end{rem}

\subsection{Twisted conjugacy} 

As we have already mentioned, one of the main ingredients in BMV-machinery is the Twisted Conjugacy Problem (\TCP). In this section we recall its definition, some previous results, and we study this problem for FATF groups. 

\begin{defn}
    Let $G$ be a group, let $x,y \in G$, and let $\varphi \in \End(G)$. We say that $x$ is \defin{$\varphi$-twisted conjugate} to $y$, denoted by $x \conj_{\varphi} y$, if there exist some $z\in G$ such that $ y=(z\varphi)^{-1}xz$. Then we also say that $z$ is a \defin{$\varphi$-twisted conjugator} of $x$ into $y$.
\end{defn}

It is immediate to check that $\varphi$-twisted conjugacy is an equivalence relation in $G$, and that $\id$-twisted conjugacy corresponds to standard conjugacy.

Let $G$ be a group given by a finite presentation, and let $\mathcal{T} \subseteq \End(G)$ be given in terms of (images of) the generators.
\begin{named}[Twisted conjugacy problem for $G$ \wrt $\mathcal{T}$, $\TCP_{\mathcal{T}}(G)$]
    Decide, given $\varphi \in \mathcal{T}$, and $x,y \in G$, whether there exists some $z\in G$ such that $ y=(z\varphi)^{-1}xz$; \ie
\begin{equation*} \label{eq: BrP}
\TCP_{\mathcal{T}}(G)
\,\equiv\,
\exists z\in G \st y=(z\varphi)^{-1}xz \?_ {\subalign{&\varphi\in\mathcal{T},\\ & x,y\in G}}
\end{equation*}
We write $\TCPa(G) = \TCP_{\!\Aut(G)}(G)$, $\TCPm(G) = \TCP_{\Mon(G)}(G)$, and $\TCPe(G) = \TCP_{\End(G)}(G)$.
\end{named}



\begin{rem}
    Note that the decidability of $\TCP$ guarantees the computability of some twisted conjugator. Namely, given $x,y \in G$, we can always recursively enumerate the elements $z\in G$ and use the positive part of $\WP(G)$ to check whether the twisted conjugacy condition $(z\varphi)^{-1} x z = y$ holds. Since such an element is known to exist, this search procedure will certainly terminate.
\end{rem}
 
The first decidability result of the $\TCP$ of free groups was obtained for automorphisms in~\cite{bogopolski_conjugacy_2006}. More recently, using the computability of the fixed subgroup proved by Mutanguha \cite{mutanguha_constructing_2022}, Logan and Ventura proved (independently) that the twisted conjugacy problem is indeed decidable for all endomorphisms of finitely generated free groups.

\begin{thm}[\cite{ventura_multiple_2021,logan_conjugacy_2023}]
\label{thm: tcp endos free}
$\TCPe(\Fn)$ is decidable. \qed
\end{thm}

Note that although $\TCPe(\Zm)$ reduces to solving a linear system of Diophantine equations and hence is clearly decidable as well (\eg reducing to the Smith Formal Form,~see \cite{norman_finitely_2012}), the decidability of $\TCPe(\Fn \times \Zm)$ is not an immediate consequence of the decidability of $\TCPe$ on its factors,
as it follows from the description of $\End(\Fn \times \Zm)$ in \Cref{prop: endos FATF}.
In order to prove it, we need to also take into account the interaction between the free and free-abelian parts arising from this description (encoded, for example in the matrix $\matr{P}$, in endomorphisms of type I).

For a group $G$ and $x\in G$, we denote by $\gamma_x$ the inner automorphism of $G$ defined by $g\mapsto x^{-1}gx$.  A convenient description of the set of twisted-conjugators follows.

\begin{lem}\label{lem: twisted conjugators}
Let $x,y\in G$, $\varphi\in \End(G)$ and $w_0\in G$ such that $(w_0^{-1}\varphi) x w_0=y$. Then, for all $w\in G$, 
\begin{equation*}
(w^{-1}\varphi) x w=y
\,\Leftrightarrow\,
w\in \Fix(\varphi\gamma_x)w_0 \, .
\end{equation*}
\end{lem}

\begin{proof}
Since $(w_0^{-1}\varphi) x w_0=y$, we have that
\begin{align*}
    (w^{-1}\varphi) x w=y
&\,\Leftrightarrow\,  (w^{-1}\varphi) x w=(w_0^{-1}\varphi) x w_0\\
& \,\Leftrightarrow\, x^{-1}((ww_0^{-1})\varphi) x=ww_0^{-1}\\
& \,\Leftrightarrow\, ww_0^{-1}\in\Fix(\varphi\gamma_{x})\\
& \,\Leftrightarrow\, w\in  \Fix(\varphi\gamma_{x})w_0 ,
\end{align*}
as we wanted to prove.
\end{proof}

\begin{thm}
$\TCPe(\Fn \times \Zm)$ is decidable.
\end{thm}
\begin{proof}
Recall that 
the problem is to decide, given $u\ta{a},v\ta{b}\in \Fn \times \Zm$, and $\Phi \in \End(\Fn \times \Zm)$ whether
\begin{equation} \label{eq: TCPe FATF}
\exists w \t[c] \in \Fn \times \Zm
\text{ such that }
((w \t[c])^{-1}) \Phi \, u\t[a] \, w \t[c]
\,=\,
v\t[b] \,.
\end{equation}

We distinguish two cases depending on the type of the input endomorphism $\Phi$ (see \Cref{prop: endos FATF}).
If $ \Phi =\Phi_{\varphi, \matr{Q},\matr{P}}$ is an endomorphism of type I, then $(w\ta{c})^{-1}\Phi=(w^{-1}\varphi)\ta{-c\matr{Q}-\vect{w}\matr{P}}$, and hence 
$(w\ta{c})^{-1}\Phi u\ta{a}w\ta{c}=(w^{-1}\varphi) u w \ta{c(\matrid[m]-\matr{Q})+a-\vect{w}\matr{P}}$.
Separating the free and free-abelian parts, the problem \eqref{eq: TCPe FATF} is reduced to deciding whether
\[
\exists w\in \Fn\,,\ \exists \vect{c}\in \Zm \text{ such that } 
\bigg{\{} \!\!
\begin{array}{l}
v=(w^{-1}\varphi) u w\\
\vect{w}\matr{P}+\vect{b}-\vect{a}=\vect{c}(\matrid[m]-\matr{Q}).
\end{array}
\]
We use \Cref{thm: tcp endos free} to decide whether there is some $w\in \Fn$ such that ${(w^{-1}\varphi)uw=v}$, and, if so, compute a twisted conjugator, say $w_0$. If there is none, then the answer to \eqref{eq: TCPe FATF} is \nop. Otherwise, according the previous lemma, the goal is to decide whether 
\begin{equation} \label{eq: aux TCP 2}
\exists w\in \Fix(\varphi\gamma_u)w_0 \text{\, such that \,} 
\vect{w}\matr{P} \in \vect{a}-\vect{b} + \im(\matrid[m]-\matr{Q}).
\end{equation}
In order to do so, we start using Mutanguha's algorithm to compute a generating set $S$ for $\Fix(\varphi\gamma_u)$, which once abelianized constitutes a (free-abelian) basis for the subgroup $\gen{S\ab} = M = (\Fix(\varphi\gamma_u))\ab \leqslant \ZZ^m$. Now it is clear that deciding \eqref{eq: aux TCP 2} is equivalent to deciding whether
\[
\exists \vect{w} \in \ZZ^m \text{ such that } 
\bigg{\{} \!\!
\begin{array}{l}
\vect{w}\in \vect{w_0} + M\\
\vect{w}\matr{P} \in \vect{a}-\vect{b} + \im(\matrid[m]-\matr{Q}),
\end{array}
\]
which is clearly decidable since it corresponds to solving a linear system of Diophantine equations.

On the other side, if $\Phi = \Phi_{z,\vect{l},\vect{h},\matr{Q},\matr{P}} \in \Endii(\Fn \times \Zm)$ then $(w\ta{c})^{-1}\Phi=z^{-\vect{c}\vect{l}^T-\vect{w}\vect{h}^T}\ta{-\vect{c}\matr{Q}-\vect{w}\matr{P}}$, and hence $(w\ta{c})^{-1}\Phi \, u\ta{a} \, w\ta{c} = z^{-\vect{c}\vect{l}^T-\vect{w}\vect{h}^T}uw \,\ta{\vect{c}(\matrid[m]-\matr{Q})-\vect{w}\matr{P}+a}$.
Separating the free and free-abelian parts, in this case the problem \eqref{eq: TCPe FATF} is reduced to deciding whether 
\begin{equation}
\exists w\in \Fn \,, \exists c\in \Zm
\text{ such that }
\bigg{\{} \!\!
\begin{array}{l}
v=z^{-\vect{c}\vect{l}^T-\vect{w}\vect{h}^T}uw\\
\vect{b}=\vect{c}(\matrid[m]-\matr{Q})-{w}\ab\matr{P}+\vect{a},
\end{array}
\end{equation}
which (after abelianizing the free part) is equivalent to deciding whether
\begin{equation}
\exists \vect{w}\in \ZZ^n \,, \exists \vect{c} \in \Zm
\text{ such that }
\bigg{\{}\!\!
\begin{array}{l}
\vect{v}=(-\vect{c}\vect{l}^T-\vect{w}\vect{h}^T)\vect{z}+ \vect{u}+\vect{w}\\
\vect{b}=\vect{c}(\matrid[m]-\matr{Q})-\vect{w}\matr{P}+\vect{a},
\end{array}
\end{equation}
a condition which corresponds, once again, to checking for solutions of a linear system of Diophantine equations, and is therefore computable. So, both (type I and type II) cases have been shown to be computable, and the proof is concluded.
\end{proof}

\begin{rem}
Notice that, for free and virtually free groups, the solution of $\TCPe$ follows from computability of the fixed subgroup of a given endomorphism, while in the case of $\FATF$ groups, there exist endomorphisms for which the fixed subgroup is not finitely generated (see e.g.~\cite{delgado_algorithmic_2013}), but $\TCPe$ is still decidable.
\end{rem}

\subsection{Two-sided Brinkmann  \&\ the conjugacy problem
}


In this section we prove that the two-sided Brinkmann problem is solvable for monomorphisms of FATF groups, and we use Logan's 2-sided variation of BMV-machinery to combine it with the decidability of $\TCP(\Fn \times \Zm)$ and  deduce the solvability of the \CP\ for ascending HNN extensions of $\Fn \times \Zm$ groups.  Although Logan proves the decidibility of $\tsBrCPe(\Fn)$ in \cite[Proposition 5.3]{logan_conjugacy_2023}, we are unaware of the existence of an algorithm to decide $\tsBrPe(\ZZ^m)$. Below, we restrict our analysis to the injective case.

Taking advantage of the description of monomorphisms given by \Cref{cor: injective bijective}, we prove a preliminary lemma.

\begin{lem}\label{lem: pushing powers conjugation}
Let $\Phi = \Phi_{\varphi,\matr{Q},\matr{P}} \in \Mon(\Fn \times \Zm)$, let $\fba{u}{a},\fba{v}{b} \in \Fn \times \Zm$, and let $r,s \in\NN$. Then,
\begin{enumerate}[ind]
    \item \label{item: l1}
    if $(\fba{u}{a})\Phi^r \conj (\fba{v}{b})\Phi^s$ then $(\fba{u}{a})\Phi^{r+1} \conj (\fba{v}{b})\Phi^{s+1}$.
    \item \label{item: l2}
    if $ u \varphi^r \conj_{\Fn} v\varphi^s$ and $(\fba{u}{a})\Phi^{r+1} \conj (\fba{v}{b})\Phi^{s+1}$, then $(\fba{u}{a})\Phi^{r} \conj (\fba{v}{b})\Phi^{s}$
\end{enumerate}
\end{lem}

\begin{proof}
The claim in $\ref{item: l1}$ is obvious (it is enough to compose with $\Phi$ on both sides). To prove~$\ref{item: l2}$, notice that two elements are conjugate in $\Fn \times \Zm$ if and only if their free parts are conjugate in $\Fn$ and their abelian parts coincide. By \eqref{eq: hom FATF I} and \eqref{eq: powers hom FATF I}:

\begin{align}
 &(\fba{u}{a}) \Phi^{r}
    \,=\,
    (u)\varphi^{r}\vect{t}^{\vect{a}\matr{Q}^{r}+ \mathbf{u}\matr{P}^{(r)}}
    \label{eq: power uta}\\
    &(\fba{v}{b}) \Phi^{s} 
    \,=\,
    (v)\varphi^{s}\vect{t}^{\vect{b}\matr{Q}^{s}+ \mathbf{v}\matr{P}^{(s)}}
    \label{eq: power vtb}
\end{align}

By hypothesis, we have that $ u \varphi^r \conj_{\Fn} v\varphi^s$, which implies that $(u\varphi^r)\ab=(v\varphi^s)\ab$, 
so we only have to prove that the abelian parts in Equations (\ref{eq: power uta}) and (\ref{eq: power vtb}) coincide. By \eqref{eq: hom FATF I} and \eqref{eq: powers hom FATF I}, we have that 
\begin{align*}
  (\fba{v}{b}) \Phi^{s+1}\tau
  \,&=  (\fba{v}{b}) \Phi^{s}\Phi\tau\\
  \,&=
  \,({(\fba{v}{b}) \Phi^{s}\tau)\matr{Q}+((\fba{v}{b}) \Phi^{s}\pi)\ab\matr{P}}\\
  \,&=\,{({\vect{b}\matr{Q}^{s}+ \mathbf{v}\matr{P}^{(s)}})\matr{Q}+(v\varphi^s)\ab\matr{P}}
  , 
\end{align*}
and similarly,
\begin{align*}
(\fba{u}{a}) \Phi^{r+1}\tau
&\,=\,{({\vect{a}\matr{Q}^{r}+ \mathbf{u}\matr{P}^{(r)}})\matr{Q}
+(u\varphi^r)\ab\matr{P}}\\
&\,=\,\,{({\vect{a}\matr{Q}^{r}+ \mathbf{u}\matr{P}^{(r)}})\matr{Q}
+(v\varphi^s)\ab\matr{P}}.
\end{align*}

Since $(\fba{u}{a}) \Phi^{r+1}\tau=(\fba{v}{b}) \Phi^{s+1}\tau$, it follows that 
$({\vect{a}\matr{Q}^{r}+ \mathbf{u}\matr{P}^{(r)}})\matr{Q}=({\vect{b}\matr{Q}^{s}+ \mathbf{v}\matr{P}^{(s)}})\matr{Q}$
and so 
${\vect{a}\matr{Q}^{r}+ \mathbf{u}\matr{P}^{(r)}}={\vect{b}\matr{Q}^{s}+ \mathbf{v}\matr{P}^{(s)}}$, by injectivity of $\matr{Q}$.
 \end{proof}

Now we can prove the two-sided version of Brinkmann's conjugacy problem for monomorphisms of $\FATF$ groups.
\begin{thm}
For any $n,m \in \NN$, $\tsBrCPm(\Fn \times \Zm)$ is decidable.

That is, (for any $m,n \in \NN$) there is an algorithm that, given $\fba{u}{a},\fba{v}{b} \in \Fn \times \Zm$ and $\Phi \in \Mon(\Fn\times\Zm)$, decides whether there exist some $r,s \in \NN$ such that  $(\fba{u}{a})\Phi^r \conj (\fba{v}{b})\Phi^s$ and, in case they exist, compute one such pair $(r,s)\in \NN^2$.

\end{thm}

\begin{proof}
 Let $\fba{u}{a},\fba{v}{b}$ and $\Phi = \Phi_{\varphi,\matr{Q},\matr{P}} \in \Mon(\Fn\times\Zm)$ be the input for $\tsBrCPm(\Fn \times \Zm)$. By~\eqref{eq: powers hom FATF I}, any $(r,s)\in \NN^2$ satisfying $(\fba{u}{a})\Phi^r \conj (\fba{v}{b})\Phi^s$ also satisfies: $u\varphi^r\sim v\varphi^s$, an instance of $\tsBrCPm(\Fn)$, which is decidable by \cite[Proposition 5.3]{logan_conjugacy_2023}. If the answer to $\tsBrCPm(\Fn)$ (on input $(u,v,\varphi)$) is \nop, then the answer to $\tsBrPm(\Fn \times \Zm)$ is also \nop. Otherwise, we compute integers $(p,q)\in \NN^2$ such that $u\varphi^p \sim v\varphi^q$.
 We now claim that, if there is a solution $(r,s)\in \NN^2$ to our problem, then there must also exist a solution of the form $(k,q)\in \NN^2$ or one of the form $(p,k)\in \NN^2$. Indeed, if there is a solution $(r,s)\in \NN^2$ with $r\leq p$, then using \Cref{lem: pushing powers conjugation} (i), we have that $(p,s+p-r)$ is also a solution. Similarly, if $s\leq q$, we have that $(r+q-s,q)$ is a solution.
 Also, if there is a solution $(r,s)\in \NN^2$ with $r>p$ and $s>q$, then we can use \Cref{lem: pushing powers conjugation} (ii) to get a solution of the desired form.
 
 Using $\CP(\Fn \times \Zm)$, we can check by brute force if there is any solution of the form $(p,k)$ with $k\in [q]$ or $(k,q)$ with $k\in[p]$. If there is one, we answer $\yep$ and output that solution. If not, we use $\BrCPm(\Fn)$ to decide whether there are $k_1$ and $k_2\in \NN$ such that $u\varphi^p\varphi^{k_1}\sim u\varphi^p$ and $v\varphi^q\varphi^{k_2}\sim v\varphi^q.$
 
 
 By \Cref{lem: philog form}, we only have to look for solutions of the form $(p,q+k_2\lambda)$ and of the form $(p+ k_1\lambda,q)$, for $\lambda\in \NN$ (if for some $i\in\{1,2\}$, $k_i$ does not exist, we take $k_i=0$).
 
 We will now see that we can decide whether there is some $\lambda\in \NN$ such that $(\fba{u}{a})\Phi^p \conj (\fba{v}{b})\Phi^{q+\lambda k_2}$, and, in case there is such $\lambda$, compute it. If $k_2=0$, then this reduces to an instance of $\CP(\Fn \times \Zm)$. If not, we only have to check whether there is some $\lambda\in \NN$ such that $(\fba{u}{a})\Phi^p\tau = (\fba{v}{b})\Phi^{q+\lambda k_2}\tau$, since the free parts are guaranteed to be conjugate.
 

 Recall that by \Cref{lem: powers conj periodic affine FATF},
$
     (\fba{v}{b})\Phi^{q+\lambda k_2}\tau = (\fba{v}{b})\Phi^{q}\tau\widetilde{T}
     ^\lambda
$, 
 where  $\widetilde T$ is the affine transformation $\widetilde T \colon \ZZ^m \to \ZZ^m$, $\vect{x} \mapsto \vect{x} \matr{Q}^{k_2} + (v\varphi^{q})\ab\matr{P}^{(k_2)}$.
     %
 
 So, our problem is reduced to finding $\lambda\in \NN$ such that $(\fba{v}{b})\Phi^{q}\tau \widetilde T^\lambda=(\fba{u}{a})\Phi^{p}\tau$, which can be done in view of \Cref{thm: KannanLipton}.
 Similarly, we can decide if there is some $\lambda\!'\in \NN$ such that $(\fba{u}{a})\Phi^{p+\lambda\!'k_1} \conj (\fba{v}{b})\Phi^{q}$. This concludes the proof.
\end{proof}

Using Logan's machinery to prove the conjugacy problem in ascending HNN extensions, we obtain the following corollary.
\begin{cor}
    The conjugacy problem is solvable for ascending HNN extensions of $\FATF$ groups.\qed
\end{cor}


\section*{Acknowledgments}
The first author was supported by
CMUP, which is financed by national funds through FCT – Fundação
para a Ciência e a Tecnologia, I.P., under the projects with reference UIDB/00144/2020
and UIDP/00144/2020.

The second author acknowledges support from the Spanish \emph{Agencia Estatal de Investigación} through grant PID2021-126851NB-I00 (AEI/FEDER, UE), as well as from the \emph{Universitat Politècnica de Catalunya} in the form of a ``María Zambrano'' scholarship.


\renewcommand*{\bibfont}{\small}
\printbibliography
\Addresses

\end{document}

